\documentclass[11pt, final, reqno]{amsart}

\usepackage[utf8]{inputenc}
\usepackage[T1]{fontenc}
\usepackage[english]{babel}
\usepackage{amsmath, amsthm}
\usepackage{ae}
\usepackage{icomma}
\usepackage{units}
\usepackage{color}
\usepackage{graphicx}
\usepackage{bbm}
\usepackage{caption}
\usepackage{array}
\usepackage[hmarginratio=1:1]{geometry}
\usepackage[hyphens]{url}
\usepackage[pdfpagelabels=false]{hyperref}
\usepackage{mathrsfs}
\usepackage{amssymb}
\usepackage{placeins}
\usepackage{tikz}
\usepackage{float}
\usepackage[nodayofweek]{datetime}
\usepackage{enumitem}
\usepackage{mathtools}
\usepackage[numbers, sort]{natbib}
\mathtoolsset{showonlyrefs=true}
\usepackage{verbatim}

\newcommand{\R}{\ensuremath{\mathbb{R}}}
\newcommand{\N}{\ensuremath{\mathbb{N}}}
\newcommand{\Z}{\ensuremath{\mathbb{Z}}}

\DeclareMathOperator{\Tr}{Tr}

\newcommand{\LT}[2]{L_{#1, #2}^{\mathrm{cl}}}
\newcommand{\eps}{\ensuremath{\varepsilon}}

\newcommand{\pps}{\hspace{0.75pt}}

\newtheorem{theorem}{Theorem}[section]
\newtheorem{definition}{Definition}
\newtheorem{lemma}[theorem]{Lemma}
\newtheorem{proposition}[theorem]{Proposition}

\numberwithin{theorem}{section}
\numberwithin{definition}{section}
\theoremstyle{remark}
\newtheorem{remark}[theorem]{Remark}

\newcommand{\limplus}{{\mathchoice{\vcenter{\hbox{$\scriptstyle +$}}}
  {\vcenter{\hbox{$\scriptstyle +$}}}
  {\vcenter{\hbox{$\scriptscriptstyle +$}}}
  {\vcenter{\hbox{$\scriptscriptstyle +$}}}
}}
\newcommand{\limminus}{{\mathchoice{\vcenter{\hbox{$\scriptstyle -$}}}
  {\vcenter{\hbox{$\scriptstyle -$}}}
  {\vcenter{\hbox{$\scriptscriptstyle -$}}}
  {\vcenter{\hbox{$\scriptscriptstyle -$}}}
}}
\newcommand{\limpm}{{\mathchoice{\vcenter{\hbox{$\scriptstyle \pm$}}}
  {\vcenter{\hbox{$\scriptstyle \pm$}}}
  {\vcenter{\hbox{$\scriptscriptstyle \pm$}}}
  {\vcenter{\hbox{$\scriptscriptstyle \pm$}}}
}}

\def\myarXiv#1#2{\href{http://arxiv.org/abs/#1}{\texttt{arXiv:#1 \hspace{-5pt}[#2]}}}

\newcommand{\optD}{{\scalebox{0.6}{$\mathcal{D}$}}}
\newcommand{\optN}{{\scalebox{0.6}{$\mathcal{N}$}}}

\begin{document}

\title[Asymptotic behaviour of cuboids optimising  Laplacian eigenvalues]{Asymptotic behaviour of cuboids optimising Laplacian eigenvalues}

\author[K. Gittins]{Katie GITTINS}
\address{Universit\'e de Neuch\^atel, Institut de Math\'ematiques, Neuch\^atel, Switzerland}
\email{katie.gittins@unine.ch}

\author[S. Larson]{Simon LARSON}
\address{KTH Royal Institute of Technology, Department of Mathematics, Stockholm, Sweden}
\email{simla@math.kth.se}

\date{\today.}

\begin{abstract}
We prove that in dimension $n \geq 2$, within the collection of unit-measure cuboids in $\R^n$ (i.e.\ domains of the form $\prod_{i=1}^{n}(0, a_n)$), any sequence of minimising domains $R_k^\optD$ for the Dirichlet eigenvalues $\lambda_k$ converges to the unit cube as $k \rightarrow \infty$. Correspondingly we also prove that any sequence of maximising domains $R_k^\optN$ for the Neumann eigenvalues $\mu_k$ within the same collection of domains converges to the unit cube as $k\to \infty$.
For $n=2$ this result was obtained by Antunes and Freitas in the case of Dirichlet eigenvalues and van den Berg, Bucur and Gittins for the Neumann eigenvalues. The Dirichlet case for $n=3$ was recently treated by van den Berg and Gittins.

In addition we obtain stability results for the optimal eigenvalues as $k \to \infty$. We also obtain corresponding shape optimisation results for the Riesz means of eigenvalues in the same collection of cuboids. For the Dirichlet case this allows us to address the shape optimisation of the average of the first $k$ eigenvalues.
\end{abstract}
\maketitle

\section{Introduction}\label{intro}

Let $\Omega \subset \R^n$, $n\geq 2$, be an open set with finite Lebesgue measure $\vert \Omega \vert < \infty$. Then the spectrum of the Dirichlet Laplace operator $-\Delta^\mathcal{D}_{\Omega}$ acting on $L^2(\Omega)$ is discrete and its eigenvalues can be written in a non-decreasing sequence, repeating each eigenvalue according to its multiplicity,
\begin{equation}\label{ei0}
\lambda_1(\Omega) \leq \lambda_2(\Omega) \leq \dots \leq \lambda_k(\Omega) \leq \dots,
\end{equation}
with $\lambda_1(\Omega) > 0$. Moreover, the sequence accumulates only at infinity.

If in addition the boundary of $\Omega$ is Lipschitz regular, then
the spectrum of the Neumann Laplace operator $-\Delta_\Omega^\mathcal{N}$ is
discrete and its eigenvalues can be written in a non-decreasing sequence, repeating each eigenvalue according to its multiplicity,
\begin{equation}
  0=\mu_0(\Omega) \leq \mu_1(\Omega) \leq \ldots \leq \mu_k(\Omega) \leq \ldots
\end{equation}
Again the sequence accumulates only at infinity.

\subsection{Optimising Laplacian eigenvalues with a measure constraint.}
For $k \in \N$ and fixed $c >0$, the existence
of sets $\Omega_k^\optD$ and $\Omega_k^\optN$ which realise the infimum respectively the supremum
in the optimisation problems
\begin{equation}\label{ei0a}
\begin{aligned}
\lambda_k(\Omega_k^\optD) &= \inf\{\lambda_k(\Omega) : \Omega \subset \R^n \mbox{ open}, \vert \Omega \vert = c\},\\
\mu_k(\Omega_k^\optN) &= \sup\{\mu_k(\Omega) : \Omega \subset \R^n\mbox{ open and Lipschitz}, \vert \Omega \vert = c\}
\end{aligned}
\end{equation}
has received a great deal of attention throughout the last century.

It was shown by Faber~\cite{Fa} in $\R^2$ and Krahn~\cite{Kr,Kr2} in any dimension that the first Dirichlet eigenvalue is minimised by the ball of measure $c$. Furthermore, Krahn~\cite{Kr2} proved that the disjoint union of two balls each of measure $\frac{c}{2}$ minimises $\lambda_2$. In the Neumann case it was shown by Szeg\H o~\cite{S54} and Weinberger~\cite{W56} that the ball of measure $c$ maximises $\mu_1$. Girouard, Nadirashvili and Polterovich~\cite{GNP} proved that amongst all bounded, open, planar, simply connected sets of area $c$, the maximum of $\mu_2$ is realised by a sequence of sets which degenerates to the disjoint union of two discs each of area $\frac{c}{2}$.

For $k \geq 3$, it is known that a minimiser of $\lambda_k$ exists in the collection of quasi-open sets, see~\cite{B1, MP}. But whether these minimisers are open is currently unresolved. In general minimisers of $\lambda_3$ are not known to date, but there are some conjectures for them, for example see~\cite{H, OV}. In the plane, and with $k \geq 5$, it is known that neither a disc nor a disjoint union of discs minimises $\lambda_k$~\cite{Be}. In addition, for some values of $k \geq 3$, numerical evidence suggests that minimisers of $\lambda_k$ might not have any natural symmetries, see~\cite{AF1}.

In the Neumann case the existence of a maximising set which realises the above supremum remains open to date (see, for example,~\cite[Subsection~7.4]{BB05}).

\subsection{Asymptotic shape optimisation}
An idea brought forward by Antunes and Freitas~\cite{MR3001382} was to consider the behaviour of minimisers of $\lambda_k$ at the other end of the spectrum. That is,
for a collection of sets in which a minimiser $\Omega_k^\optD$ of $\lambda_k$ exists for all $k \in \N$,
to determine the limiting shape of a sequence of minimising sets $(\Omega_k^{\optD})_k$ as $k \rightarrow \infty$.
Analogously, if a maximiser $\Omega_k^\optN$ of $\mu_k$ exists in some collection of sets, then
one can consider the asymptotic behaviour of a sequence of maximising sets $(\Omega_k^{\optN})_k$ as $k \rightarrow \infty$.

It was shown in~\cite{CoSo} that the statement that $\lambda_k(\Omega_k^\optD)$ resp.\ $\mu_k(\Omega_k^\optN)$ is asymptotically equal to $4\pi^2 \omega_n^{-2/n} (\tfrac kc)^{2/n}$ as $k \rightarrow \infty$, where $\omega_n$ is the measure of the unit ball in $\R^n$, is equivalent to P\'olya's conjecture: for $k \in \N$ and any bounded, open set $\Omega \subset \R^n$ of measure $c$,
\begin{equation}\label{eq:Polya_conj}
\begin{aligned}
\lambda_k(\Omega)&\geq 4\pi \Gamma\Bigl(\frac{n}{2}+1\Bigr)^{2/n} \Bigl(\frac{k}{c}\Bigr)^{2/n},\\[2pt]
\mu_k(\Omega)&\leq 4\pi \Gamma\Bigl(\frac{n}{2}+1\Bigr)^{2/n} \Bigl(\frac{k}{c}\Bigr)^{2/n}.
\end{aligned}
\end{equation}
Note that the right-hand side of~\eqref{eq:Polya_conj} is precisely the quantity we would like to find as $k\to \infty$, since $4\pi^2\omega_n^{-2/n}=4\pi \Gamma(\tfrac n2+1)^{2/n}$.
These inequalities were shown to hold for tiling domains by P\'olya~\cite{PO}, see also~\cite{Kel66}. In particular, they hold for $\Omega = \prod_{i=1}^n(0, a_i)$.

In~\cite{MR3001382}, it was shown that amongst all planar rectangles of unit area, any sequence
of minimising rectangles for $\lambda_k$ converges to the unit square as $k \rightarrow \infty$. In~\cite{vdBBG} it was shown that the corresponding result holds in the Neumann case. Furthermore, the analogous result for the Dirichlet eigenvalues in three dimensions was proven in~\cite{vdBergGittins}. That is, amongst all cuboids in $\R^3$ of unit volume, any sequence of cuboids minimising $\lambda_k$ converges to the unit cube as $k \rightarrow \infty$.
For the Dirichlet eigenvalues, it was conjectured in~\cite{AL} that the analogous result also holds in dimensions $n \geq 4$, and some support for this conjecture was obtained there (see \cite[Section~2]{AL}). Similar arguments also suggest that the corresponding result holds for the Neumann eigenvalues in dimensions $n \geq 3$ (by invoking \cite[Theorem~4]{AL} instead of \cite[Theorem~1]{AL}).

The goal of this paper is
to generalise the results of~\cite{MR3001382,vdBBG,vdBergGittins} to arbitrary dimensions. To that end,
throughout the paper we let $R=R_{a_1,\ldots,a_n}$ denote an $n$-dimensional cuboid of unit measure, that is a domain of the form $\prod_{i=1}^n(0, a_i)\subset \R^n$
where $a_{1}, \dots, a_{n} \in \R_\limplus$ are such that  $\prod_{i=1}^{n}a_{i}=1$. Without loss of generality we will always label the $a_i$ so that $a_{1} \leq \dots \leq a_{n}.$ Moreover, we let $Q$ denote the $n$-dimensional unit cube.

For $k \in \N$,
$\lambda_k(R)$ and $\mu_k(R)$ obey the two-term asymptotic formulae
\begin{equation}
  \begin{aligned}\label{ei1}
  \lambda_{k}(R) &= 4\pi\Gamma\Bigl(\frac{n}{2} + 1\Bigr)^{2/n}k^{2/n} + \frac{2\pi \Gamma(\frac{n}{2}+1)^{1+1/n}}{n\Gamma(\frac{n+1}{2})}
  |\partial R| k^{1/n} +o(k^{1/n}),\\
  \mu_{k}(R) &= 4\pi\Gamma\Bigl(\frac{n}{2} + 1\Bigr)^{2/n}k^{2/n} - \frac{2\pi \Gamma(\frac{n}{2}+1)^{1+1/n}}{n\Gamma(\frac{n+1}{2})}
  |\partial R| k^{1/n} +o(k^{1/n}),
  \end{aligned}
\end{equation}
as $k\to \infty$ (see~\cite{VI} or Section~\ref{count}).
Here, and in what follows, $|\partial R|$ denotes the perimeter of $R$. Corresponding two-term asymptotic formulae were conjectured by Weyl for more general domains $\Omega\subset \R^n$, and under certain regularity assumptions the conjecture was proven by Ivrii in~\cite{VI}.

Since the cube in $\R^n$ has smallest perimeter in the collection of $n$-dimensional cuboids,~\eqref{ei1}
suggests that the cube is the limiting domain of a sequence of optimising cuboids in this collection as $k \to \infty$.
However, this argument does not provide a proof as we are not considering a fixed cuboid $R$ and then letting $k \rightarrow \infty$. The minimising or maximising cuboids themselves depend upon $k$ (see, for instance,~\cite{AL}).

\subsection{Eigenvalues of cuboids}\label{sec:Eigenvalues of cuboids}

For a cuboid $R$ as above, the Laplacian eigenvalues are given by
\begin{equation}\label{ei2}
 \frac{\pi^2 i_1^2}{a_1^2} + \frac{\pi^2 i_2^2}{a_2^2} + \cdots + \frac{\pi^2 i_n^2}{a_n^2},
\end{equation}
where $i_1, \ldots, i_n$ are positive integers in the Dirichlet case and non-negative integers in the Neumann case.

From~\eqref{ei2} we see that a minimising cuboid of unit measure for $\lambda_k$, $k \in \N$, must exist.
Indeed, as in~\cite{vdBergGittins}, we consider a minimising sequence for $\lambda_k$ where one side-length is blowing up. Then another side-length must be shrinking in order to preserve the measure constraint. However, this shrinking side would give rise to large eigenvalues, whilst for the unit cube $Q$ we have that $\lambda_k(Q) \leq n \pi^2 k^2 < \infty$, contradicting the minimality of the sequence. To emphasise the optimality, when referring to a cuboid which minimises $\lambda_k$ we will write $R^\optD_k$ and denote its side-lengths by $a_{1,k}^*, \ldots, a_{n, k}^*$.

Similarly we see that a maximising cuboid of unit measure for $\mu_k$, $k \in \N$, exists.
As in~\cite{vdBBG}, if $(R_{\ell})_{\ell\in \N}=(R_{a_1^{\ell}, \ldots, a_n^{\ell}})_{\ell\in \N}$ is a maximising sequence for $\mu_k$ with $a_n^{\ell}\to \infty$ as $\ell \to \infty$, then for sufficiently large $\ell$
\begin{equation}
  \mu_k(R_\ell) \leq \frac{\pi^2k^2}{(a_n^{\ell})^2}
\end{equation}
and so $\mu_k(R_\ell)\to 0$ as $\ell \to \infty$. For the unit cube $Q$ we have that $\mu_k(Q) > \pi^2$, contradicting the maximality of the sequence. When referring to a cuboid which maximises $\mu_k$ we will write $R^\optN_k$ and denote its side-lengths by $a_{1,k}^*, \ldots, a_{n, k}^*$.

\subsection{Main results}
Before we state our results we need the following definition which plays a central role in what follows.
\begin{definition}\label{def1}
For $n \geq 2$, define $\theta_n$ as any exponent such that for all $a_1, \ldots, a_n \in \R_\limplus$,
\begin{equation}\label{def:theta_n}
\#\{z \in \Z^n : a_1^{-2}z_1^2 + \ldots + a_n^{-2} z_n^2 \leq t^2\} - \omega_n t^n \prod_{i=1}^n a_i = O(t^{\theta_n}), \quad \mbox{as } t \rightarrow \infty,
\end{equation}
uniformly for $a_i$ on compact subsets of\/ $\R_\limplus$.
\end{definition}

Geometrically $\theta_n$ describes the asymptotic order of growth of the difference between the number of integer lattice points in the ellipsoid $a_1^{-2}x_1^2+\ldots +a_n^{-2}x_n^2\leq t^2$ and its volume. Finding the optimal order of growth in the case $n=2$ and $a_1=a_2=1$ is the well-known, and still open, Gauss circle problem (see~\cite{IKKN04} and references therein).

If~\eqref{def:theta_n} is not required to hold uniformly for different $a_i$, then estimates for $\theta_n$ are well-known (see, for instance,~\cite{HB,H03,IKKN04}). However, with the additional requirement of a uniform remainder term the literature is less extensive. For $n\geq 5$, $\theta_n=n-2$ is known to hold and to be optimal~\cite{Gotze}. As far as the authors are aware, the smallest known value, for $n=3, 4$, is $\theta_n=\frac{n(n-1)}{n+1}$ which is due to Herz~\cite{Herz}. For $n=2$ it holds that $\theta_2\leq \tfrac{46}{73}+\eps$, for any $\eps>0$, due to Huxley~\cite{H96}. In all dimensions, $\theta_n < n-1$.

The main aim of this paper is to prove the following theorems, and thereby extend the results of~\cite{MR3001382, vdBergGittins} and~\cite{vdBBG} to all dimensions.

\begin{theorem}\label{T1}
Let $n \geq 2$. For $k \in \N$, let $R_k^\optD$ denote an $n$-dimensional unit-measure cuboid which minimises $\lambda_k$. Then, as $k \rightarrow \infty$, we have that
\begin{equation}
a_{n, k}^* =
1 + O(k^{(\theta_n-(n-1))/(2n)}),
\end{equation}
where $\theta_n$ is as defined in~\eqref{def:theta_n}.
That is, any sequence of minimising $n$-dimensional cuboids $(R_k^\optD)_k$ for $\lambda_k$ converges to the $n$-dimensional unit cube as $k \rightarrow \infty$.
\end{theorem}

\begin{theorem}\label{TN1}
Let $n \geq 2$. For $k \in \N$, let $R_k^\optN$ denote an $n$-dimensional unit-measure cuboid which maximises $\mu_k$. Then, as $k \rightarrow \infty$, we have that
\begin{equation}
a_{n, k}^* =
1 + O(k^{(\theta_n-(n-1))/(2n)}),
\end{equation}
where $\theta_n$ is as defined in~\eqref{def:theta_n}.
That is, any sequence of maximising $n$-dimensional cuboids $(R_k^\optN)_k$ for $\mu_k$ converges to the $n$-dimensional unit cube as $k \rightarrow \infty$.
\end{theorem}

A further interesting question is what this implies for the difference between $\lambda_k^* = \lambda_k(R_k^\optD)$ and $\lambda_k(Q)$, resp.\ $\mu_k^* = \mu_k(R_k^\optN)$ and $\mu_k(Q)$.
By P\'olya's inequalities~\eqref{eq:Polya_conj}, and the leading order asymptotics of $\lambda_k(Q), \mu_k(Q)$, we see that $|\lambda_k(Q)-\lambda_k^*| = O(k^{1/n})$ and $|\mu_k(Q)-\mu_k^*|= O(k^{1/n})$. By a more detailed analysis, we obtain the following.

\begin{theorem}\label{T2}
  As $k \to \infty$,
  \begin{align}
    |\lambda_k(Q)-\lambda_k^*|&=
    O(k^{(\theta_n-(n-2))/n}),\\
    |\mu_k(Q)-\mu_k^*|&=
    O(k^{(\theta_n-(n-2))/n}),
  \end{align}
  where $\theta_n$ is as defined in~\eqref{def:theta_n}.
\end{theorem}

Note that for $n\geq 5$ the above estimate states that the difference between the extremal eigenvalues and those of the unit cube remain bounded for all $k$, which we do not know to be the case for $n<5$.

\subsection{Strategy of proof}
Let $r \geq 0$ and let $R\subset \R^n$ be a cuboid of measure one. We define
\begin{equation}\label{e1}
E(r, R) := \biggl\{(x_{1}, \dots, x_{n}) \in \R^{n} : \sum_{j=1}^{n}\frac{x_{j}^{2}}{a_{j}^{2}}
\leq \frac{r}{\pi^{2}} \biggr\}.
\end{equation}
The set $E(r, R)\subset \R^n$ is an $n$-dimensional ellipsoid with radii $r_j = \frac{a_j r^{1/2}}{\pi}$, $j=1, \dots, n$, and measure $\vert E(r, R)\vert = \omega_n \prod_{j=1}^{n} r_j=\frac{\omega_n r^{n/2}}{\pi^n}$.

By \eqref{ei2}, we see that the Dirichlet eigenvalues $\lambda_1(R), \dots, \lambda_k(R)$ correspond to
integer lattice points with positive coordinates that lie inside or on the ellipsoid $E(\lambda_k(R), R)$.
In this setting, determining a cuboid of unit measure which minimises $\lambda_k$ corresponds to determining the
ellipsoid which contains $k$ integer lattice points with positive coordinates and has minimal measure.
Similarly, the Neumann eigenvalues $\mu_0(R), \mu_1(R), \dots, \mu_k(R)$ correspond to
integer lattice points with non-negative coordinates that lie inside or on the ellipsoid $E(\mu_k(R), R)$.
Determining a cuboid of unit measure which maximises $\mu_k$ corresponds to determining the
ellipsoid of maximal measure which contains fewer than $k+1$ integer lattice points with non-negative coordinates.

This observation is used to prove Theorems~\ref{T1} and~\ref{TN1} by following the strategy of~\cite{MR3001382} (see also~\cite{vdBergGittins,vdBBG}). In particular, we compare the number of lattice points that are inside or on a minimal, respectively maximal, ellipsoid to the number of lattice points that are inside or on the sphere with radius $\pi^{-1}(\lambda^*_k)^{1/2}$, respectively $\pi^{-1}(\mu^*_k)^{1/2}$, and let $k \rightarrow \infty$.
To make this comparison, we use known estimates for the number of integer lattice points that are inside or on an $n$-dimensional ellipsoid (this explains the appearance of the quantity $\theta_n$ in the above results). However, in order to use these estimates, we must first show that for any sequence of minimising or maximising cuboids, the corresponding side-lengths are bounded independently of $k$. The difficulty lies in obtaining a sufficiently good upper, resp.\ lower, bound for the Dirichlet, resp.\ Neumann, counting function which, for $\lambda, \mu \geq 0$ and $R, E(r, R)$ as above,
we define as
\begin{equation}\label{e2}
\begin{aligned}
N^\optD(\lambda, R)&:=\# \{(i_{1}, \dots, i_{n}) \in \N^{n} \cap E(\lambda, R)\},\\
N^{\optN}(\mu, R)&:=\# \{(i_{1}, \dots, i_{n}) \in (\N \cup \{0\})^{n} \cap E(\mu, R)\}.
\end{aligned}
\end{equation}

In this paper, in order to obtain an upper bound for $N^\optD(\lambda, R)$ and corresponding lower bound for $N^{\optN}(\mu, R)$, we make use of an argument going back to Laptev~\cite{Lap} and the fact that cuboids satisfy P\'olya's inequalities~\eqref{eq:Polya_conj}. This argument, together with an application of an identity due to Aizenman and Lieb (see~\cite{AizenmanLieb} or~\eqref{eq:AizenmanLieb} below), allows us to reduce the problem to estimating $\sum_{k}(\lambda-k^2)_\limplus$, which arises as the Riesz mean of the Laplacian on an interval.

The approach used in~\cite{MR3001382, vdBBG} and~\cite{vdBergGittins} to prove the two- and three-dimensional versions of Theorems~\ref{T1} and~\ref{TN1} makes use of the fact that
the functions $i \mapsto(y - i^2)^{m/2}$, for $m=1, 2,$ are concave on $[0, y^{1/2}]$.
However, for $m \geq 3$, this concavity fails and hence this approach cannot be used to deal with the higher-dimensional cases, see~\cite{vdBergGittins}. To use the same approach as in~\cite{vdBBG} to deal with the case $n = 3$, it would also be necessary to show that $\limsup_{k \to \infty} (a_{1,k}^*)^{-1} (\mu_k^*)^{-1/2} < \infty$ (compare with~\cite[Lemma~2.3]{vdBBG}). The approach taken for the Neumann case here allows us to obtain a two-term lower bound for $N^{\optN}(\mu, R)$ which enables us to avoid such considerations.
This issue was also avoided when the two-dimensional case was proven in~\cite{LL}.
Nonetheless, in any dimension it is possible to obtain a bound for the quantity $\limsup_{k\to \infty}(a_{1,k}^*)^{-1}(\mu_k^*)^{-1/2}$ by exploiting that if $a_{1,k}^*=o((\mu^*_k)^{-1/2})$ then all $\mu_l(R^\optN_k)$, for $l<k$, must be of the form $\pi^2 \sum_{j=2}^n i_j^2 (a_{j, k}^*)^{-2}$ and by the maximality of $\mu^*_k$ the domain $\prod_{j=2}^n(0, a_{j,k}^*)$ must be a maximiser of $\mu_k$ amongst cuboids in $\R^{n-1}$ of measure $1/a_{1,k}^*$.

\subsection{Additional remarks.}
Our approach naturally lifts to considering shape optimisation problems of maximising, resp.\ minimising, the Riesz means of Dirichlet, resp.\ Neumann, eigenvalues, which for $\lambda, \mu\geq 0$ and $\gamma \geq 0$ are defined by
\begin{align}
  \Tr(-\Delta^\optD_\Omega-\lambda)_\limminus^{\gamma}= \sum_{k=1}^\infty (\lambda-\lambda_k(\Omega))_\limplus^\gamma,\quad \mbox{resp.}\quad
  \Tr(-\Delta^\optN_\Omega-\mu)_\limminus^{\gamma}= \sum_{k=0}^\infty (\mu-\mu_k(\Omega))_\limplus^\gamma.
\end{align}
For $\Omega \subset \R^n$ and $\gamma \geq 3/2$ the Dirichlet case of this problem was addressed in~\cite{SL}, where it was shown that amongst collections of convex sets of unit measure, satisfying certain additional regularity assumptions, the extremal sets converge to the ball as $\lambda \to \infty$. Within the collection of $n$-dimensional cuboids we obtain the corresponding result for all $\gamma \geq 0$ in both the Dirichlet and Neumann cases, that is, any sequence of optimal cuboids converges to the unit cube as $\lambda, \mu \to \infty$ (see Propositions~\ref{prop:Riesz means Dirichlet} and~\ref{prop:Riesz means Neumann} below).

A problem which is closely related to that considered here was recently studied by Laugesen and Liu~\cite{LL}. In this article the authors consider a collection of concave, planar curves that lie in the first quadrant and have intercepts $(L, 0)$ and $(0, M)$. They fix such a curve and scale it in the $x$ direction by $s^{-1}$ and in the $y$ direction by $s$, as well as radially by $r$. Their goal is to determine the curve which contains the most integer lattice points in the first quadrant as $r \rightarrow \infty$. Under certain assumptions on the curve they prove that the optimal stretch factor $s(r) \rightarrow 1$ as $r \rightarrow \infty$. In particular, they recover the result of Antunes and Freitas~\cite{MR3001382}, and, in a similar way, that of van den Berg, Bucur and Gittins~\cite{vdBBG}.
They also obtain analogous results for $p$\pps-ellipses where $1<p<\infty$.
The case where $0<p<1$ has recently been addressed by Ariturk and Laugesen in~\cite{AL}. As mentioned above, the results of that paper lend some support to Theorem~\ref{T1} in the case where $n \geq 5$ (see~\cite[Section~2]{AL}).
Recently the case $p=1$ was treated by Marshall and Steinerberger~\cite{MS}. In contrast to the case $p\neq 1$, the set of maximising $s$
in this setting does not converge when $r\to \infty$ and in fact there is an infinite set of limit points.
After the first version of this paper appeared Marshall generalised the results of Laugesen and Liu to an $n$-dimensional setting~\cite{Marshall}. The results of that paper include the convergence results of Theorems~\ref{T1} and~\ref{TN1} as special cases.

The plan for the remainder of the paper is as follows. In Section~\ref{upperbound} we obtain bounds for the eigenvalue counting functions $N^\optD, N^{\optN}$. We continue in Section~\ref{boundedness} by applying the obtained bounds to prove that the side-lengths of a sequence of minimising, respectively maximising, cuboids $(R_k^\optD)_k$, $(R_k^\optN)_k$ are bounded independently of $k$.
In Section~\ref{count} we prove uniform asymptotic expansions for the counting functions $N^\optD(\lambda, R), N^{\optN}(\mu, R)$. All the above is combined in Section~\ref{convergence&stability} in order to prove Theorems~\ref{T1},~\ref{TN1} and~\ref{T2}. Finally, in Section~\ref{Sums_Means} we apply our methods to the shape optimisation problems of maximising, resp.\ minimising, the Riesz means of Dirichlet, resp.\ Neumann, eigenvalues and minimising the average of the first $k$ Dirichlet eigenvalues. For both problems we obtain analogous results to those obtained in the case of individual eigenvalues.

% ------------------  Prelim  --------------------

\section{Preliminaries}
We begin this section by establishing three- respectively two-term bounds for the eigenvalue counting functions for the Dirichlet and Neumann Laplacians on an arbitrary cuboid. These bounds will allow us to prove that the sequence of extremal cuboids remains uniformly bounded, i.e.\ does not degenerate, as $k$ tends to infinity (see Section~\ref{boundedness}).

We end this section by obtaining
precise and uniform asymptotic expansions for the eigenvalue counting functions on the sequence of extremal cuboids.

Here and in what follows we let $\LT{\gamma}{m}$ denote the semi-classical Lieb--Thirring constant
\begin{equation}
  \LT{\gamma}m = \frac{\Gamma(\gamma+1)}{(4\pi)^{m/2}\Gamma(\gamma+ \frac{m}{2}+1)}.
\end{equation}
For $x\in \R$ we also define the positive and negative parts of $x$ by $x_\limpm = (|x|\pm x)/2$.

% ------------------  many-term bounds  --------------

\subsection{Asymptotically sharp bounds for the eigenvalue counting functions}\label{upperbound}

In this section we prove a three-term upper bound for the counting function $N^\optD(\lambda, R)$
and a two-term lower bound for the counting function $N^\optN(\mu, R)$. More specifically we prove the following lemmas.

\begin{lemma}\label{lem:D_3term_bound}
For $n\geq 2$, there exist positive constants $c_1, c_2$ and $b_0$ such that, for any cuboid $R\subset \R^n$ with $|R|=1$, the bound
  \begin{equation}
    N^\optD(\lambda, R) \leq
     \LT{0}{n}\lambda^{n/2}-\frac{c_1 b \LT{0}{n-1}}{a_1}\lambda^{(n-1)/2}+ \frac{c_2 b^2\LT{0}{n-2}}{a_1^2}\lambda^{(n-2)/2},
  \end{equation}
  holds for all $\lambda \geq 0$ and $b\in \bigl[0, b_0\bigr]$.
\end{lemma}

\begin{lemma}\label{lem:N_2term_bound}
For $n\geq 2$, there exists $c_1>0$ such that for any cuboid $R\subset \R^n$, with $|R|=1$, the bound
  \begin{equation}
    N^\optN(\mu, R) \geq
     \LT{0}{n}\mu^{n/2} + \frac{c_1\LT{0}{n-1}}{a_1}\mu^{(n-1)/2},
  \end{equation}
  holds for all $\mu\geq 0$.
\end{lemma}

\begin{remark}
  The parameter $b$ in our bounds for $N^\optD(\lambda, R)$ allows us to tune whether we wish the bound to be more accurate near the bottom of the spectrum or asymptotically as $\lambda \to \infty$. This flexibility will be of importance for us when we prove the uniform boundedness of the extremal cuboids for the Dirichlet problem, see Section~\ref{boundedness}.

  It should be noted that for large $\lambda$ the third term is not fundamental and could be absorbed by the second one. For instance, when $n\geq 5$ a bound similar to Lemma~\ref{lem:D_3term_bound} was obtained in~\cite[Corollary~1.2]{SLpams} without the third term by instead requiring that $\lambda$ is large enough. Similarly a two-term bound in the two-dimensional case was obtained in~\cite[Proposition~10]{LL}.
  However, the procedure of lifting the above bounds to Riesz means is much simplified if the bounds are valid for all $\lambda\geq 0$, and correspondingly $\mu\geq 0$ (see Section~\ref{Sums_Means}).
\end{remark}

\begin{proof}[Proof of Lemmas~\ref{lem:D_3term_bound} and~\ref{lem:N_2term_bound}]
The main idea of the proof is to reduce the problem to proving one-dimensional estimates. To this end we follow an idea due to Laptev~\cite{Lap}, which uses
the fact that cuboids satisfy P\'olya's inequalities~\eqref{eq:Polya_conj} and the product structure of the domains.
Let $R'=(0, a_2)\times \dots \times (0, a_n)$ and write
{\allowdisplaybreaks
\begin{align}\label{eq:count_productdomain}
  N^\optD(\lambda, R) &= \sum_{k: \lambda_k(R)\leq \lambda} (\lambda-\lambda_k(R))^0 \notag\\
  &=
  \sum_{k, l : \lambda_k(R')+\lambda_l((0, a_1))\leq \lambda} (\lambda-\lambda_l((0, a_1))-\lambda_k(R'))^0 \notag\\
  &=
  \sum_{l: \lambda_l((0, a_1))\leq \lambda} \ \sum_{k: \lambda_k(R')\leq \lambda-\lambda_l((0, a_1))}((\lambda-\lambda_l((0, a_1)))-\lambda_k(R'))^0\\
  &=
  \sum_{l: \lambda_l((0, a_1))\leq \lambda} N^\optD((\lambda-\lambda_l((0, a_1)))_\limplus, R'),
\end{align}
where we use the convention ``$0^0=1$''. The above could be done with strict inequalities to avoid this issue, but to match~\eqref{e1}, \eqref{e2} we also wish to count the eigenvalues that are equal to $\lambda$.
Applying P\'olya's inequality for the counting function on $R'$, which says $N^\optD(\lambda, R') \leq \LT{0}{n-1}|R'|\lambda^{(n-1)/2}$ (see~\cite{PO}), yields that
\begin{align}\label{eq:dim_reduction_D}
  N^\optD(\lambda, R) &\leq
  \sum_{l: \lambda_l((0, a_1)) \leq \lambda} \LT{0}{n-1}|R'| (\lambda-\lambda_l((0, a_1)))_\limplus^{(n-1)/2} \notag\\
  &=
  \LT{0}{n-1}|R'| \Tr(-\Delta^\mathcal{D}_{(0, a_1)}-\lambda)_\limminus^{(n-1)/2}.
\end{align}
}

Analogously, with the only difference being that P\'olya's inequality goes in the opposite direction, one finds that
\begin{align}\label{eq:dim_reduction_N}
  N^\optN(\mu, R) &\geq
  \LT{0}{n-1}|R'| \Tr(-\Delta^\mathcal{N}_{(0, a_1)}-\mu)_\limminus^{(n-1)/2}.
\end{align}

  The Aizenman--Lieb Identity~\cite{AizenmanLieb} asserts that if $\gamma_1\geq 0$ and $\gamma_2>\gamma_1$, then, for $\eta \geq 0$,
  \begin{equation}\label{eq:AizenmanLieb}
    \Tr(-\Delta_\Omega-\eta)^{\gamma_2}_\limminus = B(1+\gamma_1, \gamma_2-\gamma_1)^{-1} \int_0^\infty \tau^{-1+\gamma_2-\gamma_1} \Tr(-\Delta_\Omega-(\eta-\tau))^{\gamma_1}_\limminus\, d\tau,
  \end{equation}
  where $B$ denotes the Euler Beta function:
  \begin{equation}
    B(x, y):= \int_0^1 t^{x-1}(1-t)^{y-1}\,dt.
  \end{equation}
  The identity follows immediately from linearity and that, for any $a\in \R$,
  \begin{equation}
    \int_0^\infty \tau^{-1+\gamma_2-\gamma_1}(a+\tau)_\limminus^{\gamma_1}\,d\tau= \int_0^{a_\limminus} \tau^{-1+\gamma_2-\gamma_1}(a+\tau)_\limminus^{\gamma_1}\,d\tau = a_\limminus^{\gamma_2}B(1+\gamma_1, \gamma_2-\gamma_1),
  \end{equation}
  by the change of variables $t=\frac{(a+\tau)_\limminus}{a_\limminus}$.

  Thus we can write the bounds~\eqref{eq:dim_reduction_D} and~\eqref{eq:dim_reduction_N} in the form
  \begin{align}\label{eq:AL_bound_D}
    N^\optD(\lambda, R) \leq \frac{\LT{0}{n-1}|R'|}{B(1+\gamma, \frac{n-1}{2}-\gamma)} \int_0^\lambda \tau^{-1+(n-1)/2-\gamma} \Tr(-\Delta^\mathcal{D}_{(0, a_1)}-(\lambda-\tau))_\limminus^{\gamma}\, d\tau,\\[4pt]
    N^\optN(\mu, R) \geq \frac{\LT{0}{n-1}|R'|}{B(1+\gamma, \frac{n-1}{2}-\gamma)} \int_0^\mu \tau^{-1+(n-1)/2-\gamma} \Tr(-\Delta^\mathcal{N}_{(0, a_1)}-(\mu-\tau))_\limminus^{\gamma}\, d\tau,\label{eq:AL_bound_N}
    \end{align}
  where we are free to choose $\gamma \in [0, (n-1)/2)$. By choosing suitable $\gamma$ and appropriate one-dimensional estimates it is possible to obtain a variety of bounds for the counting functions. The bounds that we make use of here are proven in the appendix (see Lemmas~\ref{lem:1D-Dirichlet bound} and~\ref{lem:1D-Neumann bound}).

For $n=3$ we do not use the Aizenman--Lieb Identity. Applying the bounds of Lemmas~\ref{lem:1D-Dirichlet bound} and~\ref{lem:1D-Neumann bound} to the one-dimensional traces of~\eqref{eq:dim_reduction_D} resp.~\eqref{eq:dim_reduction_N} yields the claimed bounds. For dimensions $n\geq 4$ choose $\gamma=1$ in~\eqref{eq:AL_bound_D} and~\eqref{eq:AL_bound_N} and apply Lemma~\ref{lem:1D-Dirichlet bound} resp.\ Lemma~\ref{lem:1D-Neumann bound}. Computing the resulting integrals one obtains the claimed bounds.

In the two-dimensional Neumann case a bound of the required
form was obtained in~\cite[Proposition~14]{LL}. Moreover, the two-dimensional Dirichlet case follows almost directly from Proposition~10 of the same paper. This proposition states that, for $\lambda \geq 1/a_1^2$,
\begin{equation}\label{eq:LLboundDirichlet}
  N^\optD(\lambda, R)\leq \frac{\lambda}{4\pi}- \frac{c\lambda^{1/2}}{a_1},
\end{equation}
for some constant $c>0$. We aim for a bound of the form $N^\optD(\lambda, R)\leq \frac{1}{4\pi}(\sqrt{\lambda}-b/a_1)^2$. Note that the bound is trivially true for $\lambda <\pi^2/a_1^2$.
Note also that for $ b \leq \pi$
the right-hand side is pointwise decreasing
in $b$, hence if it holds true for some $b_0$ it holds for all $b\in [0, b_0]$. Therefore, using~\eqref{eq:LLboundDirichlet} it suffices to prove that
\begin{equation}
  \frac{\lambda}{4\pi}- \frac{c \lambda^{1/2}}{ a_1}\leq \frac{\lambda}{4\pi}- \frac{b\lambda^{1/2}}{2\pi a_1}+ \frac{b^2}{4\pi a_1^2}
\end{equation}
for all $\lambda > \pi^2/a_1^2$, which is clearly true if and only if $b \leq 2\pi c$.
\end{proof}

% ------------------  extremal cuboids are unif bounded  --------------------

\subsection{Extremal cuboids are uniformly bounded.}\label{boundedness}

In this section we obtain a uniform lower bound for the shortest side-length of the extremal cuboids $R^\optD_k$ and $R^\optN_k$.

As the proof is almost precisely the same for the Dirichlet and the Neumann cases we only write out the former
in full. The only difference between the two cases is that an element of the proof in the Dirichlet case is not present in the proof of the Neumann result. This difference stems from the fact that in the Dirichlet case we have a three-term bound and so we need to bound the quantity that this extra term gives rise to.

For $n\geq 2$ let $R^\optD_k$, $k\geq 1$,
be a sequence of unit measure cuboids minimising $\lambda_k$, i.e.\ such that $\lambda_k(R^\optD_k)=\lambda_k^*$, and as usual we assume
that $a_{1, k}^*\leq \ldots \leq a_{n, k}^*$.
By optimality $\lambda_k^*\leq \lambda_k(Q)$ and so $\lambda_k^*-\eps<\lambda_k(Q)$, for any $0<\eps<1$, which implies that
\begin{equation}
 N^\optD(\lambda_k^*-\eps, Q) \leq k-1 < k \leq N^\optD(\lambda_k^*, R^\optD_k).
\end{equation}

The two-term asymptotics for the Dirichlet eigenvalue counting function on the cube (see~\cite{VI} or Section~\ref{count}) combined with Lemma~\ref{lem:D_3term_bound} then yield that
\begin{align}
  \LT{0}{n}(\lambda_k^*-\eps)^{n/2}&- \frac{\LT{0}{n-1}}{4}|\partial Q|(\lambda_k^*-\eps)^{(n-1)/2}+o((\lambda_k^*-\eps)^{(n-1)/2}) \\
&\leq \LT{0}{n}(\lambda_k^*)^{n/2} - \frac{c_1 b\LT{0}{n-1}}{a_{1, k}^*}(\lambda_k^*)^{(n-1)/2} + \frac{c_2 b^2\LT{0}{n-2}}{(a_{1, k}^*)^2}(\lambda_k^*)^{(n-2)/2}.
\end{align}
Rearranging and taking $\eps =\tfrac12$ we find that
\begin{align}
  \frac{b}{a_{1,k}^*}\biggl( 1 - \frac{c_2 b \LT{0}{n-2}(\lambda_k^*)^{-1/2}}{c_1 \LT{0}{n-1} a_{1,k}^*}\biggr) \leq \frac{n}{2 c_1} + o(1).
\end{align}
Since $\lambda_k^*=\lambda_k(R^\optD_k)\geq \lambda_1(R^\optD_k) > \pi^2(a_{1,k}^*)^{-2}$, we have that $-(\lambda_k^*)^{-1/2} \geq -a_{1,k}^*/\pi$. Hence
\begin{align}
  \frac{b}{a_{1,k}^*}\biggl(1 - \frac{c_2b \LT{0}{n-2}}{c_1\pi \LT{0}{n-1}}\biggr) \leq \frac{n}{2c_1} + o(1).
\end{align}

We now choose $b\in (0, b_0]$, where $b_0$ is as defined in Lemma~\ref{lem:D_3term_bound}, small enough so that the left-hand side is positive.
Then the above implies that there exists a $C>0$ such that
\begin{equation}\label{eb4}
\frac{1}{a_{1, k}^*} \leq C+o(1),
\end{equation}
which
in turn implies that
\begin{equation}\label{eb10}
a_{n, k}^* \leq \biggl(\frac{1}{a_{1, k}^*}\biggr)^{n-1} \leq
C^{n-1} + o(1).
\end{equation}
Thus $\liminf_{k\to \infty} a_{1,k}^* \geq 1/C >0$ and $\limsup_{k\to \infty} a_{n, k}^* <\infty$  so the side-lengths of a minimising sequence of cuboids are uniformly bounded away from zero and infinity.
For dimensions $n=2, 3$, the corresponding result was obtained, through a slightly different argument, in~\cite{MR3001382,vdBergGittins}.

To prove the corresponding result for the Neumann problem one can take the same approach. Observe that $N^\optN(\mu_k^*-\eps, R^\optN_k) \leq k-1 < k \leq N^\optN(\mu_k(Q), Q) \leq N^\optN(\mu_k^*, Q)$, for $k\geq 1$ and any $0<\eps<1$, apply the lower bound of Lemma~\ref{lem:N_2term_bound} to the left-hand side and expand the right-hand side using its two-term asymptotic expansion. Rearranging the obtained inequality yields a bound of the form~\eqref{eb4}.

% ------------------  Lattice counting arguments  --------------------

\subsection{Precise asymptotics for eigenvalue counting functions}\label{count}
Let $\lambda, \mu, r\geq 0$ and $E(r, R), N^\optD(\lambda, R)$ and $N^\optN(\mu, R)$ be as defined in Section~\ref{intro}. Assume that $R$ has bounded side-lengths so that the ellipsoid $E(r, R)$ has positive Gaussian curvature. In this section, we obtain two-term asymptotic expansions for $N^\optD(\lambda, R)$ and $N^\optN(\mu, R)$ with remainder estimates which are uniform in the side-lengths of $R$.
As the calculations for the Dirichlet and Neumann problems are almost identical, we will write out the argument in full only for the Dirichlet case and indicate what differences appear for the Neumann case.
Specifically we prove the following.

\begin{lemma}\label{lem:Asymptotic Expansion}
For $n \geq 2$ and $R= \prod_{i=1}^{n} (0, a_i) \subset \R^n$, with $a_i>0$,
\begin{align}
  N^\optD(\lambda, R) &= \LT{0}{n} |R| \lambda^{n/2}- \frac{\LT{0}{n-1}}{4} |\partial R|\lambda^{(n-1)/2}
   + O(\lambda^{\theta_n/2}),\label{em15D}\\
   N^\optN(\mu, R) &= \LT{0}{n} |R| \mu^{n/2}+ \frac{\LT{0}{n-1}}{4} |\partial R|\mu^{(n-1)/2}
   + O(\mu^{\theta_n/2}),\label{em15N}
\end{align}
as $\lambda, \mu \to \infty,$ where $\theta_n$ is as defined in~\eqref{def:theta_n}. Moreover, the remainder terms are uniform on any collection of cuboids with side-lengths contained in a compact subset of\/ $\R_\limplus$.
\end{lemma}

Similar two-term asymptotic expansions for the counting function of the Dirichlet, resp.\ Neumann, Laplacian are
known to hold for more general domains than cuboids (see, for example,~\cite{VI}). However, to obtain the orders of convergence in Theorems~\ref{T1},~\ref{TN1} and~\ref{T2} we require a better remainder estimate than what is possible in general.

\begin{proof}[Proof of Lemma~\ref{lem:Asymptotic Expansion}]
The proof  is based on the inclusion-exclusion principle. For notational simplicity, in what follows we will write $N^\optD(\lambda), N^\optN(\mu)$ and $E(r)$ with the dependence on $R$ being implicit.

By symmetry of the ellipsoid $E(r)$ we have that
\begin{equation}\label{eq:lattice split}
  \#\{\Z^n\cap E(r)\} = 2^n \#\{\N^n\cap E(r)\}+\#\{(x_1, \ldots, x_n)\in\Z^n\cap E(r): \exists i \mbox{ for which } x_i=0\}.
\end{equation}
Let $E_i(r)$ denote the set $E(r)\cap \{x_i=0\}$. As the second term in the right-hand side of~\eqref{eq:lattice split} is the union of the sets $E_i(r)\cap \Z^n$ we can apply the inclusion--exclusion principle
\begin{equation}\label{eq:inclusion-exclusion}
  \#\{\cup_{i=1}^n (E_i(r)\cap \Z^n)\}
  =
  \sum_{k=1}^n (-1)^{k+1}\biggl(\sum_{1\leq i_1<\ldots <i_k\leq n}\#\{E_{i_1}(r)\cap\ldots \cap E_{i_k}(r)\cap \Z^n\}\biggr).
\end{equation}
The set $E_{i_1}(r)\cap \ldots \cap E_{i_k}(r)$ is naturally identified with an ellipsoid in $\R^{n-k}$, namely
\begin{equation}
  E_{I}(r)=\Bigl\{(x_1, \ldots, x_{n-k})\in \R^{n-k}: \sum_{j\notin I} \frac{x_j^2}{a_j^2}\leq \frac{r^2}{\pi^2}\Bigr\},
\end{equation}
where $I=\{i_1, \ldots, i_k\}$. Moreover, we have that
\begin{equation}
  \#\{E_{i_1}(r)\cap \ldots \cap E_{i_k}(r)\cap \Z^n\}=\#\{E_I(r)\cap \Z^{n-k}\}.
\end{equation}

Since $N^\optD(R, \lambda)=\#\{\N^n\cap E(\lambda)\}$, we find from~\eqref{eq:lattice split},~\eqref{eq:inclusion-exclusion} and~\eqref{def:theta_n} that
\begin{align}
  N^\optD(R, \lambda) &= \frac{\omega_n \lambda^{n/2}}{2^n \pi^n} - \frac{\omega_{n-1}\lambda^{(n-1)/2}}{2^n \pi^{n-1}}\sum_{i=1}^n \prod_{j\neq i} a_j+ O(\lambda^{\theta_n/2}+\lambda^{\theta_{n-1}/2}+\lambda^{(n-2)/2})\\
  &=
  \LT{0}{n}\lambda^{n/2}-\frac{\LT{0}{n-1}}{4}|\partial R|\lambda^{(n-1)/2}+ O(\lambda^{\theta_n/2}).
\end{align}
In the final step we used that $2\sum_i \prod_{j\neq i}a_j = |\partial R|$ and $\theta_m\in [m-2, m-1)$ for all $m$. The uniformity of the remainder follows directly from Definition~\ref{def1}.

To obtain the corresponding expansion in the Neumann case, one writes the lattice points in $E(r)$ as the union of reflected copies of the lattice points in $E(r)\cap (\N\cup \{0\})^n$ and then applies the inclusion--exclusion principle to this union.
\end{proof}

% ---------------  Geometric convergence ---------------

\section{Geometric convergence \& spectral stability}\label{convergence&stability}
In this section, we prove Theorems~\ref{T1},~\ref{TN1} and~\ref{T2}. As the proofs of the Dirichlet and the Neumann cases are almost identical, we again write out the former case in full and indicate the differences which occur in the proof of the latter.

Since the minimisers $R_k^\optD$, respectively the maximisers $R_k^\optN$,
need not be unique, we consider an arbitrary subsequence of such extremal sets.
By the results obtained in Section~\ref{boundedness} (or the corresponding statements in~\cite{MR3001382,vdBergGittins} and~\cite{vdBBG}), we know that the extremal cuboids in any dimension are uniformly bounded in $k$, and thus the remainder terms in~\eqref{em15D} and~\eqref{em15N} are uniform with respect to $R_k^\optD$ and $R^\optN_k$, respectively.

\begin{proof}[Proof of Theorems~\ref{T1} and~\ref{TN1}]
  As in the proof of the uniform boundedness, $N(\lambda_k^*-\eps, Q) < k \leq N(\lambda_k^*, R_k^\optD)$, for any $0<\eps <1$.
  Plugging in the asymptotic expansion~\eqref{em15D}
  on both sides, we have that
  \begin{align}
      \LT{0}{n}(\lambda_k^*-\eps)^{n/2}&-\frac{\LT{0}{n-1}}{4}|\partial Q|(\lambda_k^*-\eps)^{(n-1)/2}-O((\lambda_k^*-\eps)^{\theta_n/2}) \\
      & \leq \LT{0}{n}(\lambda_k^*)^{n/2}-\frac{\LT{0}{n-1}}{4}|\partial R_k^\optD|(\lambda_k^*)^{(n-1)/2}+O((\lambda_k^*)^{\theta_n/2}). \label{Dlsqueeze}
  \end{align}
  Rearranging and choosing $\eps = \tfrac12$, we obtain that
  \begin{equation}\label{ec4}
    |\partial R_k^\optD|-|\partial Q| \leq O((\lambda_k^*)^{(\theta_n-(n-1))/2}) = O(k^{(\theta_n-(n-1))/n}),
  \end{equation}
  which, when combined with the isoperimetric inequality for cuboids,
  implies that
  \begin{equation}\label{eq:convergence}
  |\partial R_k^\optD|= \sum_{i=1}^n\frac{2}{a_{i, k}^*} = 2n + O(k^{(\theta_n-(n-1))/n}).
  \end{equation}
  By the arithmetic -- geometric means inequality, with $a_{n, k}^*=1+\delta_k > 1$, we find that
  \begin{equation}\label{eq:AMGM1}
    (n-1)(1+\delta_k)^{1/(n-1)} +\frac{1}{1+\delta_k} \leq \sum_{i=1}^n \frac{1}{a_{i, k}^*}.
  \end{equation}
  Then, by \eqref{eq:convergence} and \eqref{eq:AMGM1},
  \begin{equation}\label{eq:rate}
  (n-1)(1+\delta_k)^{n/(n-1)} + 1 \leq n + n\delta_k + O(k^{(\theta_n-(n-1))/n}).
  \end{equation}
  For each $n \geq 2$, we know by the results in Section~\ref{boundedness} (or from~\cite{MR3001382,vdBergGittins}) that there exists $T>0$ so that $\delta_k^*=a_{n, k}^*-1 \leq T$.
  Hence, letting $c(T)=\frac{(1+T)^{n/(n-1)}-1-\frac{n}{n-1} T}{T^2}>0$,  we have that
  \begin{equation}
  (1+\delta_k)^{n/(n-1)} \geq 1 + \frac{n}{n-1}\delta_k + c(T)\delta_k^2.
  \end{equation}
  By substituting this into~\eqref{eq:rate}, we deduce that $\delta_k = O(k^{(\theta_n-(n-1))/(2n)})$.

  For the Neumann case one can argue almost identically by observing (as in the proof of the uniform boundedness of $R^\optN_k$) that, for any $0 < \eps < 1$,
  \begin{equation*}
    N^\optN(\mu_k^*- \eps, R^\optN_k) \leq N^\optN(\mu_k^*, Q).\qedhere
  \end{equation*}
\end{proof}

\begin{remark}
We remark that if we restrict the collection of cuboids to a sub-collection, then the above arguments prove that any
sequence of minimising, resp.\ maximising, cuboids converges to the cuboid of smallest perimeter in this sub-collection (in particular, replace $Q$ by this cuboid in \eqref{ec4}). For example, in the sub-collection consisting of all unit-measure cuboids in $\R^n$ of the form $\prod_{i=1}^n(0, a_i)$ such that $0< a_1 \leq \dots \leq a_n$ and $c\pps a_1 = a_2$, with $c \geq 1$, any sequence of optimisers converges to the cuboid with $a_1 = c^{-(n-1)/n}$ and $a_2 = \dots = a_n = c^{1/n}$.
\end{remark}

We now turn to the question of spectral stability and the proof of Theorem~\ref{T2}.

\begin{proof}[{Proof of Theorem~\ref{T2}.}]
As in the proof of Theorems~\ref{T1} and~\ref{TN1}, we have that, for any $0<\eps <1$, $N^\optD_k(\lambda_k^*-\eps, Q)\leq k \leq N^\optD(\lambda_k^*, R^\optD_k)$. By the asymptotic expansion~\eqref{em15D} we thus find that
\begin{align}
      \LT{0}{n}(\lambda_k^*-\eps)^{n/2}&-\frac{\LT{0}{n-1}}{4}|\partial Q|(\lambda_k^*-\eps)^{(n-1)/2}-O((\lambda_k^*-\eps)^{\theta_n/2}) \\
      \leq k & \leq \LT{0}{n}(\lambda_k^*)^{n/2}-\frac{\LT{0}{n-1}}{4}|\partial R_k^\optD|(\lambda_k^*)^{(n-1)/2}+O((\lambda_k^*)^{\theta_n/2}). \label{Dlsqueezek}
  \end{align}
By the isoperimetric inequality for cuboids this
also holds with $|\partial R_k^\optD|$ replaced by $|\partial Q|$.

Choosing $\eps = \tfrac12$ yields that
\begin{equation}\label{Dlcount}
  k = \LT{0}{n}(\lambda_k^*)^{n/2} - \frac{\LT{0}{n-1}}{4}|\partial Q|(\lambda_k^*)^{(n-1)/2}+ O((\lambda_k^*)^{\theta_n/2}),
\end{equation}
as $k \to \infty$. From which we can conclude that
\begin{equation}\label{Dlopt}
  \lambda_k^* = 4\pi \Gamma\Bigl(\frac{n}{2}+1\Bigr)^{2/n}k^{2/n}
  +\frac{2\pi \Gamma(\frac{n}{2}+1)^{1+1/n}}{n \Gamma(\frac{n+1}{2})}|\partial Q|k^{1/n}+ O(k^{(\theta_n -(n-2))/n}),
\end{equation}
as $k \to \infty$. Now~\eqref{Dlcount} is the same two-term expansion as that for $N(\lambda, Q)$, so~\eqref{Dlopt} must agree with the two-term expansion for $\lambda_k(Q)$. Thus we obtain that $|\lambda_k(Q) - \lambda_k^* | = O(k^{(\theta_n -(n-2))/n})$ as $k \to \infty$.

The approach to prove the Neumann case is identical except that one instead uses that, for any $0<\eps<1$,
\begin{equation*}
    N^\optN(\mu_k^*- \eps, R^\optN_k) \leq k\leq  N^\optN(\mu_k^*, Q).
    \qedhere
\end{equation*}
\end{proof}

% ------------------  Riesz means and sums ------------------

\section{Riesz means \& eigenvalue averages}\label{Sums_Means}

Given the techniques and bounds obtained above, it is not difficult to obtain the corresponding shape optimisation results for the following problems:
\begin{enumerate}[label=(\roman*)]
  \item\label{prob:RieszDN} For $\gamma \geq 0$ and $\lambda, \mu \geq 0$,
  \begin{align}
     \sup\bigl\{\Tr(-\Delta^\mathcal{D}_R-\lambda)^\gamma_\limminus: R\subset \R^n \mbox{ cuboid}, |R|=1\bigr\},\\
     \inf\bigl\{\Tr(-\Delta^\mathcal{N}_R-\mu)^\gamma_\limminus: R\subset \R^n \mbox{ cuboid}, |R|=1\bigr\}.
  \end{align}

  \item\label{prob:averageD} For $k\in \N$,
  \begin{equation}
    \inf \Bigl\{ \frac{1}{k}\sum_{i=1}^k \lambda_i(R) : R\subset \R^n \mbox{ cuboid}, |R|=1\Bigr\}.
  \end{equation}
\end{enumerate}

For the Riesz means we prove that:
\begin{proposition}\label{prop:Riesz means Dirichlet}
  Let $n\geq 2$ and $\gamma \geq 0$. For $\lambda >0,$ let $R_\lambda^\optD$ denote any cuboid which maximises $\Tr(-\Delta_R^\optD-\lambda)_\limminus^\gamma$ amongst all cuboids $R$ of unit measure. Then as $\lambda \to \infty$ we have that
  \begin{equation}
    a_{n,\lambda}^* = 1+O(\lambda^{(\theta_n-(n-1))/4}).
  \end{equation}
\end{proposition}

\begin{proposition}\label{prop:Riesz means Neumann}
  Let $n\geq 2$ and $\gamma \geq 0$. For $\mu>0,$ let $R_\mu^\optN$ denote any cuboid which minimises $\Tr(-\Delta_R^\optN-\mu)_\limminus^\gamma$ amongst all cuboids $R$ of unit measure. Then as $\mu \to \infty$ we have that
  \begin{equation}
    a_{n,\mu}^* = 1+O(\mu^{(\theta_n-(n-1))/4}).
  \end{equation}
\end{proposition}

In~\cite{F} Freitas studied problem~\ref{prob:averageD} in the more general setting of minimising amongst all bounded, open sets of fixed measure, and obtained the leading order behaviour of the extremal values as $k \to \infty$. By utilising a connection between Riesz means of order $\gamma=1$ and the eigenvalue averages, we prove here that:
\begin{proposition}\label{prop:Extremal averages}
  Let $n\geq 2$. For $k\in \N,$ let $\,\overline{\!R}_{k}^{\optD}$ denote any cuboid which minimises the average $\frac{1}{k}\sum_{i=1}^k\lambda_i(R)$ amongst all cuboids $R$ of unit measure. Then as $k \to \infty$ we have that
  \begin{equation}
    \bar{a}_{n,k}^* = 1+O(k^{(\theta_n-(n-1))/(2n)}).
  \end{equation}
\end{proposition}

We believe that the corresponding result should also hold for the maximisation of the Neumann averages. However, we have been unable to solve an issue which appears when trying to pass from a bound for the Riesz means to a bound for the averages (see Remark~\ref{rem:Neumann averages} below).

In a similar manner as in Section~\ref{sec:Eigenvalues of cuboids} above (see also~\cite{SL}), one can conclude that for any fixed $\lambda, \mu$ or $k\in \N$ each of these problems has at least one optimal cuboid. We denote any such optimal cuboid by $R_\lambda^\optD, R_\mu^\optN$ and $\,\overline{\!R}_{k}^{\optD}$, respectively, where the bar is to distinguish from the minimisers of the individual eigenvalues.

The approach we take for~\ref{prob:RieszDN} 
is to use the Aizenman--Lieb Identity to lift our bounds for the counting functions to higher order Riesz means. For $\gamma\geq1$ this improves special cases of a pair of inequalities due to Berezin~\cite{Berezin} (see also~\cite{Lap}).
For~\ref{prob:averageD}
we use an approach based on the close relationship between the sum of eigenvalues and the Riesz means of order $\gamma=1$. This allows us to obtain a three-term bound for the sum of the first $k$ eigenvalues, which improves a special case of a bound obtained by Li and Yau~\cite{LY} (see Lemma~\ref{lem:Dirichlet_average} below).

\begin{lemma}\label{lem:3_term_bound_RieszD}
  Let $\gamma \geq 0$. There exist positive constants $c_1, c_2$ and $b_0$ such that, for any cuboid $R\subset \R^n$ with $|R|=1$, the bound
  \begin{equation}
    \Tr(-\Delta_R^\mathcal{D}-\lambda)^\gamma_\limminus \leq \LT{\gamma}{n}\lambda^{\gamma+n/2}-\frac{c_1 b \LT{\gamma}{n-1}}{a_1}\lambda^{\gamma+(n-1)/2}+ \frac{c_2 b^2 \LT{\gamma}{n-2}}{a_1^2}\lambda^{\gamma+(n-2)/2},
  \end{equation}
  holds for all $\lambda\geq 0$ and $b\in \bigl[0, b_0\bigr]$.
\end{lemma}

\begin{lemma}\label{lem:2_term_bound_RieszN}
  Let $\gamma \geq 0$. There exists $c_1>0$ such that, for any cuboid $R\subset \R^n$ with $|R|=1$, the bound
  \begin{equation}
    \Tr(-\Delta_R^\mathcal{N}-\mu)^\gamma_\limminus \geq \LT{\gamma}{n}\mu^{\gamma+n/2}+\frac{c_1\LT{\gamma}{n-1}}{a_1}\mu^{\gamma+(n-1)/2},
  \end{equation}
  holds for all $\mu\geq 0$.
\end{lemma}

\begin{proof}[Proof of Lemmas~\ref{lem:3_term_bound_RieszD} and~\ref{lem:2_term_bound_RieszN}]
  Applying the Aizenman--Lieb Identity~\eqref{eq:AizenmanLieb} with $\gamma_1=0$ and $\gamma_2=\gamma$ to both sides of Lemma~\ref{lem:D_3term_bound}, respectively Lemma~\ref{lem:N_2term_bound}, yields the result.
\end{proof}

We note that by using the Laplace transform instead of the Aizenman--Lieb Identity, one could apply the above procedure to obtain a three-term bound for $\Tr(e^{t\Delta_R^{\optD/\optN}})$ valid for all cuboids $R\subset \R^n$. Moreover, using Theorem~1.1 of~\cite{SLpams} one can obtain a tunable three-term bound (similar to Lemma~\ref{lem:3_term_bound_RieszD}) for any convex domain $\Omega \subset \R^n$ which could then, using the Laplace transform, be lifted to a corresponding bound for $\Tr(e^{t\Delta_\Omega^\optD})$. A similar inequality was obtained by van den Berg in~\cite{vdBerg} for the Dirichlet Laplacian on smooth convex domains. By using results from~\cite{SLjfa}, the upper bound of~\cite{vdBerg} can be extended to all convex domains.
\begin{lemma}\label{lem:Dirichlet_average}
  There exist positive constants $c_1, c_2$ and $b_0$ such that, for any cuboid $R\subset \R^n$ with $|R|=1$, the bound
  \begin{equation}
    \frac{1}{k}\sum_{i=1}^k \lambda_i(R) \geq \frac{4\pi n \Gamma(\frac{n}{2}+1)^{2/n}}{n+2}k^{2/n}+ \frac{c_1 b}{a_1}k^{1/n} - \frac{c_2 b^2}{a_1^2},
  \end{equation}
  holds for all $k\in \N$ and all $b\in \bigl[0, b_0\bigr]$.
\end{lemma}

\begin{proof}[Proof of Lemma~\ref{lem:Dirichlet_average}]
  It is well known that the sum of eigenvalues and the order $1$ Riesz means are related by the Legendre transform~\cite{Lap}. It is a small modification of this insight that will allow us to obtain the claimed bound from Lemma~\ref{lem:3_term_bound_RieszD} with $\gamma=1$.

  By Lemma~\ref{lem:3_term_bound_RieszD} there exist constants $c_1', c_2'>0$ such that, for any $k \in \N$,
  \begin{align}\label{eq:legendre_transform}
    \sup_{\lambda \geq 0} \Bigl( k\lambda -\sum_{i: \lambda_i\leq \lambda} (\lambda-\lambda_i(R))\Bigr)
    &\geq
    \sup_{\lambda \geq 0} \Bigl( k\lambda - \LT{1}{n}\lambda^{1+n/2}+\frac{c_1' b \LT{1}{n-1}}{a_1}\lambda^{1+(n-1)/2}\notag\\
    & \qquad- \frac{c_2' b^2 \LT{1}{n-2}}{a_1^2}\lambda^{1+(n-2)/2}\Bigr).
  \end{align}
  The supremum on the left-hand side is achieved precisely at $\lambda = \lambda_k(R)$. Indeed, the function $f_k(\lambda)=k \lambda -\sum_{i: \lambda_i\leq \lambda} (\lambda-\lambda_i(R))$ is continuous, increasing for all $\lambda$ for which $N(\lambda, R)<k$, and decreasing if $N(\lambda, R)>k$. Moreover, for $\lambda$ such that $N(\lambda, R)=k$ we have that $f_k(\lambda)=\sum_{i=1}^k \lambda_i(R)$. Thus the left-hand side reduces to
  \begin{equation}
     \sup_{\lambda \geq 0} \Bigl( k\lambda -\sum_{i: \lambda_i\leq \lambda} (\lambda-\lambda_i(R))\Bigr) = \sum_{i=1}^k \lambda_i(R).
  \end{equation}

  On the other hand, maximising the right-hand side of the inequality is slightly more difficult and there may also be a question of uniqueness of the maximum. However, on this side we may choose any $\lambda \geq 0$ and still obtain a valid inequality.

  Choosing $\lambda$ to maximise $k\lambda - \LT{1}{n}\lambda^{1+n/2}$, which corresponds to
  \begin{equation}
    \lambda = \biggl(\frac{k}{(\frac{n}2+1)\LT{1}{n}}\biggr)^{2/n} = 4\pi \Gamma\Bigl(\frac{n}{2}+1\Bigr)^{2/n}k^{2/n},
  \end{equation}
  ensures that the leading order term has the sharp constant (this follows from the equivalence, via the Legendre transform, of the Li--Yau inequality for the sum of eigenvalues and the Berezin inequality for the Riesz mean of order $\gamma=1$, see~\cite{Lap}). With the above choice of $\lambda$ we obtain the claimed bound from~\eqref{eq:legendre_transform}.
\end{proof}

\begin{remark}\label{rem:Neumann averages}
If one attempts to apply the same technique as above to obtain a lower bound for the average of the Neumann eigenvalues from Lemma~\ref{lem:2_term_bound_RieszN}, the inequality after the Legendre transform is reversed. Therefore one cannot pick $\mu$ analogously to how we chose $\lambda$ above. Instead one needs to prove an upper bound for
\begin{equation}
  \sup_{\mu \geq 0} \Bigl(k\mu - \LT{1}{n}\mu^{1+n/2}-\frac{c_1\LT{1}{n-1}}{a_1}\mu^{1+(n-1)/2}\Bigr),
\end{equation}
which is sufficiently good to obtain the uniform boundedness of the extremal cuboids.
\end{remark}

\subsection{{Proof of Propositions~\ref{prop:Riesz means Dirichlet}--\pps\ref{prop:Extremal averages}}}

With the above bounds in hand, and almost step-by-step following the proof in Section~\ref{boundedness}, or the corresponding proof in~\cite{SL}, one obtains that $R^\optD_\lambda, R^\optN_\mu$ and $\,\overline{\!R}^\optD_k$ are uniformly bounded as $\lambda, \mu$ or $k$ goes to infinity.

For the Riesz means, in both the Dirichlet case and the Neumann case, the proof is completely analogous to that in Section~\ref{boundedness} by using Lemmas~\ref{lem:3_term_bound_RieszD} and~\ref{lem:2_term_bound_RieszN} and
the asymptotic expansions one obtains from Lemma~\ref{lem:Asymptotic Expansion} via the Aizenman--Lieb Identity.

For the eigenvalue averages we require an upper bound for $\frac{1}{\overline{a}_{1, k}k^{1/n}}$, which can be obtained as follows.
Since $\,\overline{\!R}^\optD_k$ is a minimiser, we have that
\begin{align}
  \frac{k\pi^2}{\overline{a}_{1, k}^2} \leq k\lambda_1(\,\overline{\!R}^\optD_k)\leq \sum_{i=1}^k \lambda_i(\,\overline{\!R}^\optD_k)\leq \sum_{i=1}^k \lambda_i(Q).
\end{align}
Inserting that, as $k\to \infty$,
\begin{equation}
\sum_{i=1}^k\lambda_i(Q)= \frac{4\pi n \Gamma(\frac{n}{2}+1)^{2/n}}{n+2}k^{1+2/n}
+ \frac{2\pi \Gamma(\frac{n}{2}+1)^{1+1/n}}{(n+1)\Gamma(\frac{n+1}{2})}|\partial Q|k^{1+1/n} + o(k^{1+1/n})
\end{equation}
and rearranging implies the required bound.

To find an asymptotic expansion for the eigenvalue averages, one can make use of the corresponding two-term expansions that we have for $\lambda_i(R)$ and calculate the asymptotics of the resulting sums (for instance using the Euler--Maclaurin formula).

In a similar manner as in the preceding section, for these problems one could also obtain estimates for the spectral stability, i.e.\ to what order in the respective parameters do the extremal eigenvalue means or averages approach those of the limiting domain $Q$. However, by finer analysis of the asymptotics, and not lifting the results for the counting function, it should be possible to obtain sharper estimates than what is obtained directly by the method outlined in the previous paragraph. This is due to the fact that in the above problems the erratic behaviour of the eigenvalues and counting function has in some sense been reduced by summing.

It is possible to analyse the asymptotic behaviour of the extremal averages of the first $k$ Neumann eigenvalues amongst unit-measure cuboids by invoking Theorem~\ref{T2}. Indeed, by using that
\begin{equation}
  \frac{1}{k}\sum_{i=0}^k \mu_i(Q) \leq \sup\Bigl\{\frac{1}{k}\sum_{i=0}^k \mu_i(R): R\subset \R^n \mbox{ cuboid}, |R|=1\Bigr\} \leq \frac{1}{k}\sum_{i=0}^k \mu_i^\optN
\end{equation}
and Theorem~\ref{T2}, one obtains precise two-term asymptotics for the extremal averages, and finds that they agree with the corresponding asymptotics for $Q$. However, as mentioned above we have been unable to obtain an inequality which is sharp enough to conclude that the sequence of extremal cuboids for this problem remains uniformly bounded as $k \to \infty$. Thus our approach yields nothing about the geometric convergence.

\medskip\noindent{\bf Acknowledgements.} KG was supported by the Swiss National Science Foundation grant no.\ 200021\_163228 entitled \emph{Geometric Spectral Theory}.
SL acknowledges support from the Swedish Research Council grant no.\ 2012-3864, and also wishes to thank Bruno Colbois for the invitation to visit the Universit\'e de Neuch\^atel. The authors wish to thank the referee for
his/her very helpful comments. They also extend their thanks to
Sinan Ariturk, Michiel van den Berg, Pedro Freitas and Richard Laugesen for encouraging discussions, as well as the organisers of the conference \emph{Shape Optimization and Isoperimetric and Functional Inequalities}, CIRM, where this work was initiated.

\appendix
\section{One-dimensional bounds}\label{Appendix}

\begin{lemma}\label{lem:1D-Dirichlet bound}
There exist constants $c_1, c_2, b_0>0$ such that, for all $\lambda\geq 0$ and $a>0$,
\begin{equation*}
  \Tr(-\Delta^\mathcal{D}_{(0, a)}-\lambda)_\limminus = \sum_{k\geq 1}\Bigl(\lambda-\frac{\pi^2k^2}{a^2}\Bigr)_\limplus \leq a\LT{1}{1} \lambda^{3/2} - b c_1\LT{1}{0} \lambda + \frac{ b^2c_2}{a} \LT{1}{-1}\lambda^{1/2},
\end{equation*}
for all $b \in [0, b_0]$.
\end{lemma}

\begin{lemma}\label{lem:1D-Neumann bound}
  There exists $c_1>0$ such that, for all $\mu\geq 0$ and $a>0$,
  \begin{equation*}
    \Tr(-\Delta^{\mathcal{N}}_{(0, a)}-\mu)_\limminus= \sum_{k\geq 0}\Bigl(\mu-\frac{\pi^2k^2}{a^2}\Bigr)_\limplus \geq a \LT{1}{1} \mu^{3/2} + c_1\LT{1}{0}\mu.
  \end{equation*}
\end{lemma}
\begin{remark}
  For our purposes it is essential that the leading order term agrees with the asymptotic one. The lower order terms are of less importance up to their behaviour in $\lambda$ and $a$. However, in the Dirichlet case it is important that the third term can be dominated by the second one by choosing $b$ sufficiently small.

  We also emphasise that when applying the Aizenman--Lieb Identity~\eqref{eq:AizenmanLieb} it simplifies matters if we have bounds valid for all $\lambda, \mu\geq 0$. This is the reason for proving the above inequalities for $\lambda, \mu \geq 0$ even though our main interest here is focused on large $\lambda, \mu$.
\end{remark}

\begin{proof}[Proof of Lemma~\ref{lem:1D-Dirichlet bound}]
  By rescaling it suffices to prove that, for $\lambda\geq 0$ and small enough $b$,
  \begin{equation*}
    \sum_{k\geq 1}(\lambda-k^2)_\limplus \leq \frac{2}{3}\lambda^{3/2}-b c_1 \lambda+\frac{4b^2 c_2}{\pi}\lambda^{1/2}.
  \end{equation*}
  We will prove this with $c_1=\frac{4}{3}$, $c_2=\frac{\pi}{6}$ and $b \leq 1- \frac{1}{6}\sqrt{\frac{27+\sqrt{3}}{2}}.$

  With $r=\sqrt{\lambda}-\lfloor \sqrt{\lambda}\rfloor$ we have that
  \begin{equation*}
    \sum_{k\geq 1}(\lambda-k^2)_\limplus = \frac{2}{3}\lambda^{3/2}-\frac{\lambda}{2}+\Bigl(r-r^2-\frac{1}{6}\Bigr)\lambda^{1/2}+ \frac{1}{6}(r-3r^2+2r^3).
  \end{equation*}
  Maximising the coefficient in front of $\lambda^{1/2}$ and the constant term with respect to $r\in [0, 1)$ we obtain
  \begin{equation}\label{eq:large_lambda}
    \sum_{k\geq 1}(\lambda-k^2)_\limplus
    \leq
     \frac{2}{3}\lambda^{3/2}-\frac{\lambda}{2}+\frac{\lambda^{1/2}}{12}+ \frac{1}{36\sqrt{3}}.
  \end{equation}

  We aim for a bound of the form $\sum_{k\geq 1}(\lambda-k^2)_\limplus\leq \frac{2}{3}\lambda^{1/2} (\sqrt{\lambda}-b)^2$, which holds for all $\lambda\geq 0$ and some $b>0$. Note that this bound holds trivially for all $\lambda \leq 1$, and thus we only need to choose $b$ so that it is valid for all $\lambda > 1$. Moreover, note that, for $b<1$ and $\lambda > 1$, this bound is pointwise decreasing in $b$. Hence if we know the bound to hold for some $b_0$ then it holds for all $0\leq b \leq b_0$.

  Since we have an upper bound in terms of the polynomial in~\eqref{eq:large_lambda}, it suffices to choose $b$ so that, for all $\lambda > 1$,
  \begin{align*}
  \lambda^{3/2}- \frac{3}{4}\lambda+\frac{\lambda^{1/2}}{8} 
  + \frac{1}{24\sqrt{3}} \leq \lambda^{1/2} (\sqrt{\lambda}-b)^2 = \lambda^{3/2}- 2b \lambda+b^2\lambda^{1/2}.
  \end{align*}
  Rearranging we see that this is equivalent to
  \begin{equation*}
    \Bigl(\frac{1}{8}-b^2\Bigr)\lambda^{1/2}+ \frac{1}{24\sqrt{3}} \leq \Bigl(\frac{3}{4}-2b\Bigr)\lambda,
  \end{equation*}
  and thus we must choose $b< 3/8$. If this is true then, since $\lambda > 1$,
  \begin{equation*}
    \Bigl(\frac{3}{4}-2b\Bigr)\lambda \geq \Bigl(\frac{3}{4}-2b\Bigr)\lambda^{1/2}.
  \end{equation*}
  Thus it is sufficient to choose $b$ satisfying
  \begin{equation*}
     \Bigl(\frac{1}{8}-b^2\Bigr)\lambda^{1/2}+ \frac{1}{24\sqrt{3}} \leq \Bigl(\frac{3}{4}-2b\Bigr)\lambda^{1/2},
  \end{equation*}
  or equivalently so that
  \begin{equation*}
    \frac{1}{24\sqrt{3}}\leq \Bigl(\frac{5}{8}-2b+b^2\Bigr)\lambda^{1/2}.
  \end{equation*}
  This holds for all $\lambda >1$ if and only if the inequality is valid at $\lambda=1$. Thus we can choose $b\in [0, b_0]$ with $b_0=1- \frac{1}{6}\sqrt{\frac{27  + \sqrt{3}}{2}}<\frac{3}{8}.$
\end{proof}

\begin{proof}[Proof of Lemma~\ref{lem:1D-Neumann bound}]
  We shall prove that the claimed bound holds if and only if $c_1\leq\frac{36-\sqrt{3}}{108}$. By scaling it is sufficient to prove that
  \begin{equation}\label{eA.2N}
    \sum_{k\geq 0} (\mu-k^2)_\limplus \geq \frac{2}{3}\mu^{3/2}+c_1 \mu.
  \end{equation}
  Analogously to the Dirichlet case above
  \begin{equation*}
    \sum_{k\geq 0} (\mu-k^2)_\limplus = \frac{2}{3}\mu^{3/2}+\frac{\mu}{2}+\Bigl(r-r^2-\frac{1}{6}\Bigr)\mu^{1/2}+ \frac{1}{6}(r-3r^2+2r^3),
  \end{equation*}
  where $r:=\sqrt{\mu}-\lfloor \sqrt{\mu}\rfloor$. Minimising the coefficient in front of $\mu^{1/2}$ and the constant term with respect to $r\in [0, 1)$, we find that
  \begin{equation*}
    \sum_{k\geq 0} (\mu-k^2)_\limplus \geq \frac{2}{3}\mu^{3/2}+\frac{\mu}{2}-\frac{\mu^{1/2}}{6}-\frac{1}{36\sqrt{3}}.
  \end{equation*}
  For $\mu\geq 1$ it is easy to prove that
  \begin{equation*}
    \frac{2}{3}\mu^{3/2}+\frac{\mu}{2}-\frac{\mu^{1/2}}{6}-\frac{1}{36\sqrt{3}}\geq \frac{2}{3}\mu^{3/2}+c_1 \mu,
  \end{equation*}
  if and only if $c_1 \leq \frac{36-\sqrt{3}}{108}$.

  What remains is to prove that the bound is valid for $\mu \in [0, 1)$. In this range the inequality~\eqref{eA.2N} reduces to
  \begin{equation*}
    \mu \geq \frac{2}{3}\mu^{3/2}+c_1 \mu.
  \end{equation*}
  As the right-hand side is strictly convex and the bound is valid at $\mu=0$ and $\mu=1$ the proof is complete.
\end{proof}

% ------------------  Bibliography  --------------------

\medskip
%References
\bibliographystyle{amsplain}

\end{document}